\documentclass[a4paper,12pt]{article}
\usepackage{latexsym}
\usepackage{amsmath}
\usepackage{amssymb}
\usepackage{amscd}
\usepackage{graphicx}
\newtheorem{theo}{Theorem}[section]
\newtheorem{lemma}[theo]{Lemma}

\newtheorem{definition}[theo]{Definition}
\newcommand{\EQ}{\begin{equation}}
\newcommand{\EN}{\end{equation}}
\newenvironment{proof}{\indent{\em Proof.  }}{\hfill{$\Box$}\bigskip}

\newcommand{\Aut}{\mbox{\rm Aut}}

\newcommand{\supp}{\mbox{\rm supp}}

\newcommand{\wt}{\mbox{\rm wt}}

\newcommand{\ba}{{\bf a}}

\newcommand{\bh}{{\bf h}}
\newcommand{\bc}{{\bf c}}

\newcommand{\by}{{\bf y}}
\newcommand{\bx}{{\bf x}}

\newcommand{\bv}{{\bf v}}
\newcommand{\bu}{{\bf u}}
\newcommand{\bw}{{\bf w}}
\newcommand{\bo}{{\bf 0}}

\newcommand{\F}{\mathbb{F}}

\newcommand{\CC}{{\mathcal{C}}}

\title{New families of completely transitive
codes and distance transitive graphs\footnote{This work has been
partially supported by the Spanish MICINN grants MTM2009-08435;
the Catalan grant 2009SGR1224 and also by the
Russian fund of fundamental researches 12-01-00905. \newline \indent $^1$Department of Information and Communications Engineering,
 Universitat Aut\`{o}noma de Barcelona,  08193-Bellaterra, Spain
 (email:~josep.rifa@uab.cat, joaquim.borges@uab.cat).\newline \indent $^2$ Institute for Problems of Information Transmission, Russian Academy of Sciences, Bol'shoi Karetnyi per. 19, GSP-4, Moscow, 127994, Russia (e-mail:\, zinov@iitp.ru).}}

\author{J. Borges$^1$, J. Rif\`{a}$^1$, V. A. Zinoviev$^2$}

\begin{document}

\maketitle

\begin{abstract}
In this paper new infinite families of linear binary
completely transitive codes are presented. They have covering
radius $\rho = 3$ and $4$, and are a half part of the binary
Hamming and the binary extended Hamming code of length
$n=2^m-1$ and $2^m$, respectively, where $m$ is even. From these new
completely transitive codes, in the usual
way, i.e., as coset graphs, new presentations of infinite families of
distance transitive coset graphs of
diameter three and four, respectively, are constructed.
\end{abstract}


\section{Introduction}\label{int}

We use the standard notation $[n,k,d]$ for a binary linear code $C$
of length $n$, dimension $k$ and minimum distance $d$ over the
binary field, which will be denoted by $\F$.

The automorphism
group $\Aut(C)$ of $C$ consists of all $n\times n$ binary permutation
matrices $M$, such that $\bc M \in C$ for all $\bc \in C$.  Note
that the automorphism group $\Aut(C)$ coincides with the subgroup
of the symmetric group $S_n$ consisting
of all $n!$ permutations of the $n$ coordinate positions which send
$C$ into itself. $\Aut(C)$ acts in a natural way over the set of cosets
of $C$: $\pi(C+\bv)=C+\pi(\bv)$ for every $\bv\in\F^n$ and $\pi\in\Aut(C)$.

For a code $C$, with minimum distance $d$, denote by
$e = \lfloor (d-1)/2 \rfloor$ its {\em packing radius}. For a vector
$\bx \in \F^n$ denote by $\wt(\bx)$ its Hamming weight (i.e., the number
of its nonzero positions), and by $\wt_i(\bx)$ its weight reduced modulo
$i$, i.e., $\wt_i(\bx) = \wt(\bx) \pmod{i}$, where $i \geq 2$.
For two vectors $\bx=(x_1,\ldots, x_n)$ and $\by = (y_1, \ldots, y_n)$
from $\F^n$ denote by $d(\bx, \by)$ the Hamming distance between $\bx$
and $\by$ (i.e., the number of positions $i$, where $x_i \neq y_i$).

Given any vector $\bv \in \F^n$ its {\em distance} to the code
$C$ is
\[
d(\bv,C)=\min_{\bx \in C}\{ d(\bv, \bx)\}
\]
and the {\em covering radius} of the code $C$ is
\[
\rho=\max_{\bv \in \F^n} \{d(\bv, C)\}.
\]
Clearly $e\leq\rho$ and $C$ is a {\em perfect} code if and only if
$e=\rho$.

For a given binary code $C$ with covering radius
$\rho=\rho(C)$ define
\[
C(i)~=~\{\bx \in \F^n:\;d(\bx,C)=i\},\;\;i=1,2,\ldots,\rho,
\]
and
\[
C_i~=~\{\bc \in C:\;\wt(\bc) = i\},\;\;i=0,1,\ldots, n
\]

Let $J = \{1,2, \ldots, n\}$ be the set of coordinate positions of
vectors from $\F^n$. Denote by $\supp(\bx)$ the {\em support} of
the vector $\bx = (x_1, \ldots, x_n)$, i.e.
\[
\supp(\bx)~=~\{j \in J:~x_j \neq 0\}.
\]
Say that a vector $\bx$ {\em covers} a vector $\by$ if
$\supp(\by)\subseteq \supp(\bx)$.
Two vectors $\bx$ and $\by$ are {\em neighbors}, if $d(\bx, \by) =
1$.

\begin{definition}\label{de:1.1}\cite{del,neu}  A code $C$
with covering radius $\rho=\rho(C)$ is completely regular,
if for all $l\geq 0$ every vector $x \in C(l)$ has the same
number $c_l$ of neighbors in $C(l-1)$ and the same number
$b_l$ of neighbors in $C(l+1)$.
\end{definition}

Also, define $a_l = (q-1){\cdot}n-b_l-c_l$ and note that
$c_0=b_\rho=0$. Define the intersection  array of $C$ as
$(b_0, \ldots, b_{\rho-1}; c_1,\ldots, c_{\rho})$.

\begin{definition}\cite{giu,sole} A binary linear code $C$ with
covering radius $\rho$ is completely transitive if $\Aut(C)$
has $\rho + 1$ orbits when acts on the cosets of $C$.
\end{definition}

Since two cosets in the same orbit should have the same weight
distribution, it is clear that any completely transitive code is
completely regular.



We need some facts on distance regular graphs, where we follow
\cite{bro2}. Let $\Gamma$ be a finite connected simple (i.e.
undirected, without loops and multiple edges) graph. Let
$d(\gamma, \delta)$ be the distance between two vertices
$\gamma$ and $\delta$ (i.e., the number of edges in the minimal
path between $\gamma$ and $\delta$). Denote
\[
\Gamma_i(\gamma) ~=~\{\delta \in \Gamma:~d(\gamma, \delta) = i\}.
\]

Two vertices $\gamma$ and $\delta$ from $\Gamma$ are {\em neighbors}
if $d(\gamma, \delta) = 1$.

An {\em automorphism} of a graph $\Gamma$ is a
permutation $\pi$ of the vertex set of $\Gamma$ such that, for all
$\gamma, \delta \in \Gamma$ we have $d(\gamma,\delta)=1$, if and
only if $d(\pi\gamma,\pi\delta)=1$. Let $\Gamma_i$ be a subgraph
of $\Gamma$ with the same vertices, where an edge $(\gamma, \delta)$
is defined when the vertices $\gamma, \delta$ are at distance $i$
in $\Gamma$. The graph $\Gamma$, with diameter $\rho$, is called {\em primitive} if it is
connected and all $\Gamma_i$ ~($i=1, \ldots, \rho$) are connected, and
{\em imprimitive} otherwise.

\begin{definition}\cite{bro2}\label{14:de:1.3}
A simple connected graph $\Gamma$ is called {\em distance regular},
if it is regular of valency $k$, and if for any two vertices
$\gamma, \delta \in \Gamma$ at distance $i$ apart, there are precisely
$c_i$ neighbors of $\delta$ in $\Gamma_{i-1}(\gamma)$ and $b_i$
neighbors of $\delta$ in $\Gamma_{i+1}(\gamma)$.
Furthermore, this graph is called {\em distance transitive}, if
for any pair of vertices $\gamma, \delta$ at distance
$d(\gamma, \delta)$ there is an automorphism $\pi$ from
$\mbox{\rm Aut}(\Gamma)$ which move this pair to any other given
pair $\gamma', \delta'$  of vertices at the same distance
$d(\gamma, \delta) = d(\gamma', \delta')$.
\end{definition}

The sequence ${\bf{I}}(\Gamma) = (b_0, b_1, \ldots, b_{\rho-1};
c_1, c_2, \ldots, c_\rho)$, where $\rho$ is the diameter of $\Gamma$,
is called the {\em intersection array} of $\Gamma$. The numbers
$c_i, b_i$, and $a_i$, where $a_i=k- b_i - c_i$, are called
{\em intersection numbers}. Clearly $b_0 = k,~~b_\rho = c_0 = 0,~~c_1 = 1.$

Let $C$ be a linear completely regular code with covering radius $\rho$ and
intersection array  $(b_0, \ldots , b_{\rho-1}; c_1, \ldots c_{\rho})$.
Let $\{D\}$ be a set of cosets of $C$. Define the graph $\Gamma_C$, which is
called the {\em coset graph of $C$}, taking all different cosets $D = C+ {\bf x}$ as
vertices, with two vertices $\gamma = \gamma(D)$ and $\gamma' = \gamma(D')$
adjacent, if and only if the cosets $D$ and $D'$ contain neighbor vectors,
i.e., there are ${\bf v} \in D$ and ${\bf v}' \in D'$ such that $d({\bf v}, {\bf v}') = 1$.

\begin{lemma}\cite{bro2,ripu}\label{lem:2.5}
Let $C$ be a linear completely regular code with covering radius $\rho$ and
intersection array  $(b_0, \ldots , b_{\rho-1}; c_1, \ldots c_{\rho})$
and let $\Gamma_C$ be the coset graph of $C$. Then $\Gamma_C$ is
distance regular of diameter $\rho$ with the same intersection array.
If $C$ is completely transitive, then $\Gamma_C$ is distance transitive.
\end{lemma}

Completely regular and completely transitive codes is a classical
subject in algebraic coding theory, which is closely connected with
graph theory and algebraic combinatorics. Existence and enumeration
of all such codes are  open hard problems (see \cite{bro2,del,neu}
and references there). Even for the case $q\,n \leq 48$ (suggested
by Neumaier \cite{neu}) these problems are quite far to be completed.

Perfect codes and, in particular, Hamming codes, induce many
different families of completely regular codes with fixed or growing
covering radius (see  \cite{bor1,bor2,bor3,rif1,rif2} and references
there). In a recent paper \cite{bor3} we described completely regular
codes which are halves of a binary Hamming code, obtained by adding
one row to the parity check matrix of a Hamming code. As a result
we obtained three new infinite families of linear binary completely
regular codes with covering radius $\rho = 3$ and $4$.

The purpose of this paper is to prove that all completely
regular codes constructed in \cite{bor3} are completely transitive. This is proved in Section \ref{sec:CTC}.
In the usual way, i.e., as coset graphs, we will see in Section \ref{sec:DTG} that new infinite families of
completely transitive codes induce new presentations of infinite families of
distance transitive coset graphs of diameter three and four.

\section{Completely transitive codes}\label{sec:CTC}

Let $H_m$ denote the parity check matrix of the Hamming code
$\mathcal{H}_m$ of length $n=2^m-1$, where the column $\bh_i$ of
$H_m$ is the binary representation of $\alpha^i$, $i=0,1,...,n-1$,
through the polynomial base  $\alpha^0, \alpha^1,..., \alpha^{n-1}$. For a
given even $m \geq 4$ and any  $i_1, i_2 \in \{0,1,2,3\}$, where
$i_1 \neq i_2$, denote
by $\bv_{i_1,i_2} = (v_0,v_1, ..., v_{n-1})$ the binary vector
whose $i$-th position $v_i$ is a function of the weight
of the column $\bh_i$:
\begin{equation}\label{vi1i2}
v_i~=~\left\{
\begin{array}{ccc}
1, ~~&\mbox{if}&~~\wt(\bh_i) \equiv i_1~\mbox{or}~ i_2 \pmod{4},\\
0, ~~&  &  ~~\mbox{otherwise}.
\end{array}
\right.
\end{equation}

Vector $\bv_{i_1,i_2}$ can be seen as a boolean function  $f_{i_1,i_2}$ over $\F_2^{m}\backslash \{\bo\}$, where:
$$
f_{i_1,i_2}(\bx)=~\left\{
\begin{array}{ccc}
1, ~~&\mbox{if}&~~\wt(\bx) \equiv i_1~\mbox{or}~ i_2 \pmod{4}\\
0, ~~&  &  ~~\mbox{otherwise}.
\end{array}
\right.
$$

It is well known~\cite{macw} that the automorphism group $\Aut(\mathcal{H}_m)$
of the Hamming code (i.e., the set of all permutations of
coordinates which fix the code) is isomorphic to the general linear group
$GL(m,2)$ of all the $m\times m$ nonsingular matrices over $GF(2)$. This group $\Aut(\mathcal{H}_m)$ acts 2-transitively over the set
of coordinate positions $J$ (or columns of $H_m$) and has more
powerful transitivity properties. It is well known, for example, that given any pair of
ordered sets of $m$  positions (corresponding to independent column vectors
in $H_m$), there exists a permutation in $\Aut(\mathcal{H}_m)$ moving one set to
the other one.

\begin{theo}\label{theo:3.0}
Assume that $i_1 - i_2 \equiv 1 \pmod{2}$ and let $H_m$ be the parity
check matrix of the Hamming code
$\mathcal{H}_m$ of length $n=2^m-1$, where $m$ is even, and let
$H_m(\bv_{i_1,i_2})$ be obtained from $H_m$ by adding one more row
$\bv_{i_1,i_2}$ given by (\ref{vi1i2}). Let $\mathcal{C}=\mathcal{C}_{i_1,i_2}$ be the
$[n, n - m - 1, 3]$ code with parity check matrix
$H_m(\bv_{i_1,i_2})$. Then:

The group $\Aut(\CC)$ coincides with the symplectic group $Sp(m,2)$ and its order is:
\[
|\Aut(\CC)|= (2^m-1){\cdot}(2^{m-1}){\cdot}(2^{m-2}-1){\cdot}(2^{m-3})\cdots (2^2-1){\cdot}(2).
\]
\end{theo}

\begin{proof}
In~\cite[Th. 2.2]{bent} it was proved that for any even $m$, $m\geq 4$, the function $f_{i_1,i_2}$ is quadratic for $i_1 - i_2 \equiv 1 \pmod{2}$. In these cases we have:
\begin{equation}\label{theo:2.2}
\begin{split}
f_{2,3}(\bx) &= \bx Q \bx^T, \\
f_{1,2}(\bx) &= \bx Q \bx^T + L\bx^T,\\
f_{0,1}(\bx) &= \bx Q \bx^T + \epsilon,\\
f_{0,3}(\bx) &= \bx Q \bx^T + L\bx^T +\epsilon,
\end{split}
\end{equation}
where $Q$ is the all-one upper triangular binary $m\times m$ matrix with zeroes in
the diagonal, $L$ is the all-one binary vector of length $m$, $\epsilon = 1$ and $\bx\in \F^m$.

Associated to $f_{i_1,i_2}$ there is a symplectic form~\cite[Ch. 15. \S 2]{macw} defined by:
\begin{equation}
{\cal B}(\bu,\bv) = f_{i_1,i_2}(\bu+\bv) +f_{i_1,i_2}(\bu)+f_{i_1,i_2}(\bv)+\epsilon,
\end{equation}
where $\bu,\bv\in \F^m$ and $\epsilon = 1$ or $\epsilon = 0$ when $0\in \{i_1,i_2\}$ or $0\notin \{i_1,i_2\}$, respectively.

From~\cite{macw} we know that, $\Aut({\mathcal C}) = \Aut({\mathcal C}^\perp)$
and so, a permutation of the coordinate positions (so, the elements in $J$)
represented by the $n\times n$ matrix $P$ is in $\Aut({\mathcal C})$ if and
only if $H_m(\bv_{i_1,i_2})P$ is again a parity check matrix for code
${\mathcal C}$.  Moreover, following~\cite{mac61}, the above condition happens
if and only if  $H_m(\bv_{i_1,i_2})$ and  $H_m(\bv_{i_1,i_2})P$ are related by
a linear transformation of coordinates, which we denote by $K$. This is the
key point. This means that finding the automorphism group  $\Aut({\mathcal C})$
is reduced to find all the nonsingular $m\times m$ matrices $K$ preserving the symplectic form ${\cal B}$. So, such that ${\cal B}(K\bu,K\bv)={\cal B}(\bu,\bv)$. Hence, the automorphism group $\Aut(C)$ is isomorphic to the symplectic group $Sp(m,2)$. The order of this symplectic group is a well known result. This proves the statement.
\end{proof}

Now, we are ready to prove that the group $\Aut({\mathcal C})$ acts
transitively over ${\mathcal C}(i)$, for any $i\in \{0,1,2,3\}$.

\begin{theo}\label{theo:3.1}
Let $m$ be even, let $i_1 -i_2 = 1 \pmod 2$, let $H_m$ be the parity check
matrix of the Hamming code $\mathcal{H}_m$ of length $n=2^m-1$, and let
$H_m(\bv_{i_1,i_2})$ be
obtained from $H_m$ by adding one more row $\bv_{i_1,i_2}$.
Let $\mathcal{C}=\mathcal{C}_{i_1,i_2}$ be the $[n, n - m - 1, 3]$ code
with parity check matrix
$H_m(\bv_{i_1,i_2})$. Then, $\mathcal{C}$ is a completely transitive code.
\end{theo}

\begin{proof}
It is known from~\cite{bor3} that $\CC$ is a completely regular code with covering radius 3. Now, we have to prove that under action of $\Aut({\mathcal C})$ all
cosets of $\mathcal{C}$ are partitioned into $4$ orbits. Since
there is only one coset of weight $3$ we have to consider only
cosets of weights $1$ and $2$.

Consider cosets of weight $1$. It is clear from the construction of $\Aut(\CC)$ that there exists an automorphism swapping any two columns in $H_m$ and so, moving a coset of weight 1 to any other one.

Consider cosets of weight $2$, say, $D={\mathcal C}+\bx$,
where $\wt(\bx)=2$. Let $\supp(\bx)=\{j_1,j_2\}$. Since
$\bx$ is not covered by codewords from ${\mathcal C}_3$,
we conclude that $f_{i_1,i_2}(\bh_{j_1}+\bh_{j_2})\not= f_{i_1,i_2}(\bh_{j_1})+f_{i_1,i_2}(\bh_{j_2})$ and so, ${\cal B}(\bh_{j_1},\bh_{j_2}) \not= \epsilon$.
But, as $\Aut(\CC)$ is constructed, any pair of columns $\bh_{j_1},\bh_{j_2}$ such that ${\cal B}(\bh_{j_1},\bh_{j_2}) \not= \epsilon$ can be moved to any
other pair, with the same property, by some element in $\Aut({\mathcal C})$. Therefore any coset of weight $2$
can be moved by the action of $\Aut({\mathcal C})$ to any other
coset of weight $2$.
\end{proof}

As a consequence, we have the following result, which
strengthen the corresponding result in \cite{bor3}.

\begin{theo}\label{theo:3.2}
Let $m$ be even, let $H_m$ be the parity check matrix of
the Hamming code $\mathcal{H}_m$ of length $n=2^m-1$, and
let $H_m(\bv_{i_1,i_2})$ be obtained from $H_m$ by adding
one more row $\bv_{i_1,i_2}$. Let $\mathcal{C}=\mathcal{C}_{i_1,i_2}$
be the $[n,n-m-1,3]$ code with parity check matrix $H_m(\bv_{i_1,i_2})$.
\begin{itemize}
\item If $\{i_1,i_2\} = \{0,1\}$ or $\{0,3\}$, then
$\mathcal{C}$ is a non antipodal completely transitive code
with covering radius $\rho=3$ and intersection array
$(n, (n-3)/2, 1; 1, (n-3)/2, n)$.
\item If $\{i_1,i_2\} = \{1,2\}$ or $\{2,3\}$, then $\mathcal{C}$
is an antipodal completely transitive code
with covering radius $\rho=3$ and intersection array
$(n, (n+1)/2, 1; 1, (n+1)/2, n)$.
\item If $\{i_1,i_2\} = \{0,2\}$, then $\mathcal{C}$
is an even part of the Hamming code, i.e., a completely transitive
$[n,k-1,4]$ code with covering radius $\rho=3$.
\item If $\{i_1,i_2\} = \{1,3\}$, then $\mathcal{C}$
is the Hamming code $\mathcal{H}_m$.
\end{itemize}
\end{theo}

\begin{proof}
For the first two cases we have $i_1-i_2 \equiv 1 \pmod{2}$.
Hence, by Theorem \ref{theo:3.1}, these codes are completely
transitive and completely regular.

For the case $\{i_1,i_2\} = \{0,2\}$, note that adding all the rows of
$H_m(\bv_{i_1,i_2})$ gives the all-one vector. Therefore ${\mathcal C}$
is the set of codewords of even weight of the Hamming code. Such code is
completely regular and completely transitive \cite{bor1}.

Finally, if $\{i_1,i_2\} = \{1,3\}$, then adding all the rows of
$H_m(\bv_{i_1,i_2})$ gives the all-zero vector, i.e., we are leaving
${\mathcal C}$ without changes.
\end{proof}

Consider the extended codes from the ones obtained above.
We give one
lemma from \cite{bor3}, about dual weights of codes
$\mathcal{C}^*_{i_1,i_2}$. By ${\cal H}^*_m$ we denote an
extended Hamming code of length $2^m$ and by $v_{i_1,i_2}^*$ the extended vector of $v_{i_1,i_2}$.

\begin{lemma}\label{weights}\cite{bor3}
Let $m$ be even. The weight distribution of the coset
$\bv^*_{i_1,i_2} + ({\cal H}^*_m)^{\perp}$ is:
\begin{itemize}
\item $2^{m-1}\pm~2^{\frac{m}{2}-1}$, when $i_1-i_2=1 \pmod 2$.
\item $\{0,2^{m-1}\}$, when $\{i_1,i_2\}=\{1,3\}$.
\item $\{2^{m-1},2^m\}$, when $\{i_1,i_2\}=\{0,2\}$.
\end{itemize}
\end{lemma}

Note that it is not the same to extend the code $\mathcal{C}_{i_1,i_2}$ or to add a new row $v_{i_1,i_2}$ to the parity check matrix of the extended code $\CC^*$. The next lemma will show us the difference.

\begin{lemma}\label{ext}
Let $i_1-i_2 \equiv 1 \pmod{2}$. We have that $(\mathcal{C}_{i_1,i_2})^* = (\CC^*)_{i_1+1,i_2+1}$ if and only if $0\notin\{i_1,i_2\}$.
\end{lemma}
\begin{proof}
Adding the row $v_{i_1,i_2}$ given by (\ref{vi1i2}) to matrix $H_m$ we obtain a parity check matrix for $\CC_{i_1,i_2}$. Extending this code we obtain the same code as the one obtained after adding the row $v_{i_1+1,i_2+1}$ to the parity check matrix $H_m^*$ (matrix $H^*_m$ is obtained from $H_m$ adding a zero column and, later, the all-one row). The point is that this row $v_{i_1+1,i_2+1}$ is of even weight (Lemma~\ref{weights}) and so, has a zero in the parity check position if and only if $1\notin \{i_1+1,i_2+1\}$ or the same, if $0\notin \{i_1,i_2\}$.
\end{proof}

From \cite{bor3} we know that the code
$\mathcal{C}^*_{i_1,i_2}$ is completely regular if and only if
$0\notin\{i_1,i_2\}$. In this case, since Lemma~\ref{ext}, is the same to refer to the extension of $\CC_{i_1,i_2}$ or to refer to $(\CC^*)_{i_1+1,i_2+1}$. Furthermore, the automorphism group of the extension depends on this situation.

\begin{theo}\label{autext}
Let $i_1-i_2 \equiv 1 \pmod{2}$ and $0\notin\{i_1,i_2\}$. Let $\CC^*$ be the extended code of $\CC=\CC_{i_1,i_2}$. Then:
 \begin{enumerate}
 \item[(i)] $\Aut(\mathcal{C}^*) = \Aut(\mathcal{C})\ltimes \F^m$.
 \item[(ii)] The order of $\Aut(\mathcal{C}^*)$ is:
\[
(2^m-1){\cdot}(2^{m-1}){\cdot}(2^{m-2}-1){\cdot}(2^{m-3})\cdots
(2^2-1){\cdot}(2)\cdot 2^m.
\]
 \end{enumerate}
\end{theo}
\begin{proof}
Let $\bh_1,\cdots,\bh_n \in \F^m$ be the columns of $H_m$, where $n=2^m-1$, and let $\bh_0$ be the zero vector in $\F^m$. Vector $\bh_i$ represent the $i$th coordinate positions of the codewords in $\CC$ and also in $\CC^*$ (assuming that the parity check position corresponds to vector $\bh_0$). For any $\bv\in F^m$, let $T_\bv: \F^m \longrightarrow \F^m$ be the translation on $\F^m$ defined by $T_\bv(\bx)=\bx+\bv$, for any $\bx\in \F^m$. We can also think of $T_\bv$ as acting on $\CC^*$ by permuting the coordinates of the codewords in $\CC^*$. The $i$th coordinate goes to the $j$th coordinate, such that $\bh_j=\bv+\bh_i$.
 As all the codewords in $\CC^*$ have even weight it is clear that $T_\bv$ is in $\Aut(\CC^*)$. Indeed, let $\ba=(a_0,\ldots,a_n)\in \CC^*$. This means that $\sum_{i=0}^n a_i \bh_i =\bo$. Now, $\sum_{i=0}^n a_i (\bh_i+\bv) =
 \sum_{i=0}^n a_i \bv =\bo$ and $T_\bv(\ba)\in \CC^*$.

 Furthermore, $T_m=\{T_\bv\,:\,\bv\in \F^m\}$ is a normal subgroup in $\Aut(\CC^*)$. Indeed, for any $\phi\in \Aut(\CC^*)$ we have that $\phi T_\bv\phi^{-1}$ is again a translation $T_\bw$, where $\bw=\phi(\bv)$. For any element $\alpha\in\Aut(\CC^*)$, it is clear that we can find $\alpha'\in\Aut(\CC^*)$ fixing the extended coordinate and a vector $\bu\in \F^m$, such that $\alpha=\alpha'T_{\bu}$. Therefore, we have $\Aut(\CC^*)/T_m \cong \Aut(\CC)$ and so $\Aut(\CC^*)$ is the semidirect product of $\F^m$ and $\Aut(\CC)$ (obviously, we can identify $T_m$ with $\F^m$). The first statement is proven.

The second statement is a direct consequence of Theorem~\ref{theo:3.0} and $|\F^m|=2^m$.
\end{proof}

\begin{theo}\label{theo:3.3}
Let $\mathcal{C}^*_{i_1,i_2}$ be the $[n+1, n - m - 1, 4]$ code of
length $n+1 = 2^m$ obtained by extending
$\mathcal{C}_{i_1,i_2}$. Then, $\mathcal{C}^*_{i_1,i_2}$ is
completely transitive with $\rho=4$ and intersection array $(n+1, n,
\frac{n+1}{2}, 1; 1, \frac{n+1}{2}, n, n+1)$ if and only if
$0\notin\{i_1,i_2\}$.
\end{theo}

\begin{proof} From \cite{bor3} we know that the code
$\mathcal{C}^*_{i_1,i_2}$ is completely regular if and only if
$0\notin\{i_1,i_2\}$ and also we know the intersection array for these cases.
Therefore, if this condition is not satisfied, the code is
not completely regular and neither completely transitive.
Hence, we have to prove that the completely regular code
$\mathcal{C}^* = \mathcal{C}^*_{i_1,i_2}$ is completely
transitive. To do so, we prove that all the cosets with the same minimum weight are in the same orbit by the action of $\Aut(\CC^*)$.

The number of cosets of $\CC^*$ are twice the cosets of $\CC$. Let $0$ mean
the parity check position, i.e.,
$J^* = \{0\} \cup J$. If $\bv+\CC$ is a coset of $\CC$, where $\bv$ is a representative vector of minimum weight then $(0|\bv)+\CC^*$ and $(1|\bv)+\CC^*$ are cosets of $\CC^*$.  There is only one coset of $\mathcal{C}^*$  of weight $4$, namely $(1|\bv)+\CC^*$, where $\bv+\CC$ is the only coset of weight three in $\CC$. Clearly this coset is fixed under the action of $\Aut(\mathcal{C}^*)$.

Now consider the cosets of $\mathcal{C}^*$ of weight $r\in\{1,2,3\}$.  They are of the form $(0|\bv)+\CC^*$, where $\bv+\CC$ is a coset of weight $r$ of $\CC$ and of the form $(1|\bv)+\CC^*$, where $\bv+\CC$ is a coset of weight $r-1$ of $\CC$. Cosets of the same minimum weight in $\CC$ can be moved among them by $\Aut(\CC)$ and so, as $\Aut(\CC)\subset \Aut(\CC^*)$ we need only to show that there exist an automorphism in $\Aut(\CC^*)$ moving $(0|\bv)+\CC^*$ to $(1|\bv')+\CC^*$, where $\bv,\bv'$ are at distance $r$ and $r-1$ from $\CC$, respectively. We further assume that $\supp(\bv')\subset \supp(\bv)$ and so, $\supp(\bv)=\supp(\bv')\cup \{i\}$, for some $i\in J$. The existence of the wanted automorphism is straightforward from Theorem~\ref{autext}, The automorphism $T_{\bh_i}$ (see the proof of Theorem~\ref{autext}) moves $(0|\bv)+\CC^*$ to $(1|\bv'')+\CC^*$, where $\supp(\bv'')=\{j+i\,:\,j\in \supp(\bv')\}$ and, finally, by using an automorphism from $\Aut(\CC)$ we can move from $\bv''+\CC$ to $\bv'+\CC$.
\end{proof}

\section{On antipodal distance transitive coset graphs\\ of diameter $3$ and $4$}\label{sec:DTG}

Denote by $\Gamma_{i_1,i_2}$ (respectively, $\Gamma^*_{i_1,i_2}$)
the coset graph, obtained from the code $\mathcal{C}_{i_1,i_2}$
(respectively $\mathcal{C}^*_{i_1,i_2}$) by Lemma \ref{lem:2.5}.
From Theorems \ref{theo:3.2} and \ref{theo:3.3} we obtain the
following result, which gives a new description, as coset graphs, of some
known graphs.

The next theorem was stated in~\cite{bor3} without explaining the property of transitivity of such graphs that we include here.

\begin{theo}\label{grafs}
For any even $m$, ~$m\geq 4$ there exist imprimitive and antipodal
distance transitive coset graphs $\Gamma_{0,1}$, ~$\Gamma_{1,2}$ with $v=2^{m+1}$ vertices and
$\Gamma^*_{1,2}$ with $v=2^{m+2}$ vertices. Specifically:
\begin{itemize}
\item $\Gamma_{0,1}$ has the intersection array $(n,
\frac{n-3}{2}, 1;1, \frac{n-3}{2}, n).$
\item $\Gamma_{1,2}$ has
the intersection array $(n, \frac{n+1}{2}, 1;1, \frac{n+1}{2},
n).$ \item $\Gamma^*_{1,2}$ has the intersection array
$(n+1,n,\frac{n+1}{2},1;1,\frac{n+1}{2},n,n+1).$
\item The graphs $\Gamma_{0,1}$ and $\Gamma_{1,2}$ are $Q$-polynomial.
\end{itemize}
\end{theo}

All the graphs $\Gamma_{0,1}$, $\Gamma_{1,2}$ and $\Gamma^*_{1,2}$
are known. The first two graphs provide antipodal distance regular
$2$-covers $(2^m,2,2^{m-1}-2)$ and $(2^m,2,2^{m-1})$ \cite{god},
respectively, of the complete graph $K_{2^m}$~\cite[Sec. 12.5]{bro2}).
The graph $\Gamma_{1,2}$ has been constructed by Thas \cite{tha}. His
construction involves a nondegenerate quadric in $PG(n,2)$. The equivalent
construction using a nondegenerate symplectic form has been done by
Brouwer \cite{bro2,god}.

Since any distance regular graph on $2(k+1)$ vertices with intersection
array $(k,\mu,1;1,\mu,k)$ is a Taylor graph~\cite{tai}, both graphs $\Gamma_{0,1}$
and $\Gamma_{1,2}$ belong to this class. The graphs $\Gamma_{0,1}$ and
$\Gamma_{1,2}$ are exactly the case (iv) in \cite[Sec. 7.6.C]{bro2},
where all distance transitive Taylor graphs are given.

The graph $\Gamma^*_{1,2}$ is a Hadamard graph of
order $n+1$. This graph is unique in the sense that any distance regular
graph with such parameters is a Hadamard graph~\cite[Sec. 1.8]{bro2}. It is also known that these graphs are distance transitive \cite[Th. 7.6.5]{bro2}.

It is well known \cite[Prop. 8.2.3]{bro2} that
an antipodal distance regular graph of diameter $\rho=3$ is
$Q$-polynomial if and only if it has an intersection array
$(k, \mu, 1; 1, \mu, k)$.

\end{document}